\documentclass[11pt,reqno,a4paper]{amsart}
	\usepackage{amsmath,amssymb,amsthm,framed,cancel,tikz,hyperref}
\newtheorem{theorem}{Theorem}[section]
\newtheorem{lemma}[theorem]{Lemma}
\newtheorem{problem}[theorem]{Problem}
\newtheorem{proposition}[theorem]{Proposition}
\newtheorem{corollary}[theorem]{Corollary}

\theoremstyle{definition}

\theoremstyle{remark} 
\theoremstyle{remark} 
\theoremstyle{remark}

\newcommand{\N}{\mathbb{N}}

\newcommand{\R}{\mathbb{R}}

\newcommand{\Id}{\operatorname{Id}}
\newcommand{\complemented}{\stackrel{c}{\hookrightarrow}}
\newcommand{\dotcup}{\ensuremath{\mathaccent\cdot\cup}}
\theoremstyle{remark}
\newtheorem{remark}[theorem]{Remark}

\numberwithin{equation}{section}

\begin{document}
\title[Complementations in $C(K,X)$ and $\ell_\infty(X)$]{Complementations in $C(K,X)$ and $\ell_\infty(X)$}
\author{Leandro Candido}
\address{Federal de S\~ao Paulo, ICT-UNIFESP, S\~ao Jos\'e dos Campos-SP, Brazil}
\email{\texttt{leandro.candido@unifesp.br}}+
\thanks{}




\begin{abstract} 
We investigate the geometry of $C(K,X)$ and $\ell_{\infty}(X)$ spaces through complemented subspaces
of the form $\left(\bigoplus_{i\in \varGamma}X_i\right)_{c_0}$. Concerning the geometry of $C(K,X)$ spaces we extend some results of 
D. Alspach and E. M. Galego from \cite{AlspachGalego}. On $\ell_{\infty}$-sums of Banach spaces we prove that if $\ell_{\infty}(X)$ has a complemented subspace 
isomorphic to $c_0(Y)$, then, for some $n \in \N$, $X^n$ has a subspace isomorphic to $c_0(Y)$. We further prove the following: 
\begin{itemize}
    \item[(1)] If $C(K)\sim c_0(C(K))$ and $C(L)\sim c_0(C(L))$ and $\ell_{\infty}(C(K))\sim \ell_{\infty}(C(L))$, then $K$ and $L$ have the same cardinality. 
    \item[(2)]if $K_1$ and $K_2$ are infinite metric compacta, then $\ell_{\infty}(C(K_1))\sim \ell_{\infty}(C(K_2))$ if and only if $C(K_1)$ is isomorphic to $C(K_2)$.
\end{itemize}

\end{abstract}

\maketitle

\section{Introduction}

In this paper we are mainly interested in investigating the geometry of Banach spaces of the form $C(K,X)$, the space of continuous functions defined on a Hausdorff compactum $K$ with values in a Banach space $X$, and the geometry of the space of bounded sequences in a Banach space $X$, $\ell_{\infty}(X)$. More generally speaking, our investigation lies in the class $C_b(K,X)$ of the bounded continuous functions defined in a Hausdorff space $K$ with values in $X$.   

Our research in this field can be divided into two interrelated parts. The first part of our investigation concerns $C(K,X)$ spaces and goes back to the classical and celebrated Cembranos-Freniche theorem \cite{Cem} which states that $C(K,X)$ has a 
complemented subspace isomorphic to $c_0$ whenever $K$ is an infinite Hausdorff compactum and $X$ is an infinite dimensional Banach space. The Cembranos-Freniche theorem in the past decades has influenced many lines of research, for recent examples  \cite{CK} \cite{ViniGale} \cite{CGS}, and was extended in many directions \cite{SaabSaab} \cite{Ryan} \cite{GaHa}. Galego and Hagler in \cite{GaHa}, among other results, isolated conditions on $K$ and $X$ implying that $C(K,X)$ will have a complemented subspace isomorphic to $c_0(\varGamma)$, where $\varGamma$ is an infinite set, not necessarily countable. In this paper, by using some ideas of \cite{AlspachGalego} and \cite{GaHa} we obtain yet another Cembranos-Freniche type theorem, see Theorem \ref{mainTheorem1}, that generalizes Galego-Hagler's aforementioned result for zero-dimensional compacta.

The second part of our investigation is more concerned with the isomorphic theory of $\ell_{\infty}(C(K))$ spaces. 
For Hausdorff compacta $K$ and $L$, a natural question that arises is which properties are transferred from $L$ to $K$ if  
$C(K)$ is isomorphic to $C(L)$. For recent striking contributions in this setting see \cite{Plebanek1} and \cite{Plebanek2}. In the context of 
$\ell_{\infty}$-sums of Banach spaces, our investigation revolves around the following related problem:

\begin{problem}\label{mainproblem}
If $\ell_{\infty}(C(K))$ is isomorphic to $\ell_{\infty}(C(L))$ and $L$ has some property $\mathcal{P}$. Does $K$ have property $\mathcal{P}$?
\end{problem}

The class of $\ell_\infty$-sums of Banach spaces is very subtle concerning the geometry of its members. On the one hand it may have several distinct complemented subspaces as proved by W. B. Johnson \cite{BillJohnson}: there exists a sequence $(X_n)_n$ of finite dimensional Banach spaces such that if $X$ is a separable Banach space, then $\left(\bigoplus_{n\in \N} X_n\right)_{\ell_{\infty}}$ has a complemented subspace isometric to $X^*$. On the other hand,
if $X$ is a finite dimensional space, then $\ell_{\infty}(X)$ is isomorphic to $\ell_\infty$ and by a classical result of Lindenstrauss \cite{Lind}, any infinite dimensional complemented subspace of $\ell_{\infty}(X)$ is isomorphic to $\ell_{\infty}$.

By using ideas from Khmyleva \cite{Khmyleva}, we isolate some conditions that will allow us to address the Problem \ref{mainproblem}. In one of our main results in this direction, we prove that whenever $K_1$ and $K_2$ are infinite metric compacta, then $\ell_{\infty}(C(K_1))$ is isomorphic to $\ell_{\infty}(C(K_2))$ if and only if $C(K_1)$ is isomorphic to $C(K_2)$. This relates to another result of Galego \cite[Corollary 5.8]{galego1}. 

The paper is organized as follows. In Section \ref{Sec-Terminology} we fix some general terminology to be used along the paper. In Section \ref{Sec-C(K,X)} we will establish our Cembranos-Freniche type theorem and present some applications. In Section \ref{Sec-ellinfty(X)} we conduct our investigation on the isomorphic theory of $\ell_{\infty}(C(K))$ spaces and present our main results in this setting. Finally, in Section \ref{Sec-FurtherApp} we address a related problem that involves results from both Sections \ref{Sec-C(K,X)} and \ref{Sec-ellinfty(X)}.

\section{Terminology}
\label{Sec-Terminology}

For a compact Hausdorff space $K$ and a Banach space $X$, $C(K,X)$ denotes the Banach space of all continuous functions $f:K\to X$, equipped with the norm $\|f\|=\sup_{t\in K}\|f (t)\|$. When $X=\R$, this space will be denoted by $C(K)$. Given a family of Banach spaces $(X_i)_{i\in \varGamma}$ we denote by $\left(\bigoplus_{i\in \varGamma}X_i\right)_{\ell_\infty}$ and $\left(\bigoplus_{i\in \varGamma}X_i\right)_{c_0}$ their $\ell_{\infty}$-sum and $c_0$-sum respectively. When $X_i=X$ for each $i\in \varGamma$, these sums will be denoted by $\ell_{\infty}(\varGamma,X)$ and $c_0(\varGamma,X)$ respectively.
If moreover $\varGamma=\N$ (and $X=\R$), we write $\ell_{\infty}(X)$ and $c_0(X)$ ($\ell_{\infty}$ and $c_0$).

Whenever $K$ is a Tychonov space, $\beta K$ denotes its Stone-\v Cech compactification. We denote $\beta \N^*=\beta\N\setminus \N$. It is well known that $C(\beta \N^*)\cong \ell_{\infty}/c_0$. For ordinal spaces, that is, ordinals endowed with the usual order topology, we will use the classical interval notation. 
The cardinality of any set $\varGamma$ will be denoted by $|\varGamma|$. The first infinite cardinal will be denoted by $\omega$, the first uncountable cardinal by $\omega_1$ and the cardinality of the continuum bu $\mathfrak{c}$. For topological spaces $K$ and $L$, we write 
$K\approx L$ to indicate that they are homeomorphic.

For Banach spaces $X$ and $Y$, a bounded linear operator $T : X\to Y$ is said to be a
linear embedding of $X$ into $Y$ if there is $C>0$ such that  $C\|u\| \leq \|T(u)\|$ for all $u \in X$. We will often write $X\hookrightarrow Y$ to indicate that such embedding exists. If a linear embedding $T:X\to Y$ is also surjective we say that $X$ and $Y$ are isomorphic and write $X\sim Y$, and if additionally $T$ is an isometry we write  $X\cong Y$. 

A bounded linear operator $P:X\to X$ is said to be a projection if $P^2(u)=P(u)$ for all $u\in X$. In this case we say that the image $P[X]$ is a complemented subspace of $X$.
We will often write $Y\complemented X$ to indicate that $Y$ is isomorphic to a complemented subspace of $X$. The following elementary fact will play an important role in this paper. 

\begin{proposition}\label{fact1}
A Banach space $Y$ has a complemented subspace isomorphic to $X$ if and only if there are bounded linear operators $T:X\to Y$ and $S:Y\to X$ such that $S\circ T = \Id_X$ is the identity on $X$. In this case $T$ is a linear embedding and  $T\circ S$ is a projection onto $T[Y]$.
\end{proposition}

We will also adopt other standard notational conventions. For a Banach space $X$, $B_X$ stands for its closed unit ball, $S_X$ its sphere and $X^*$ its (topological) dual space. 
Any other standard terminology for Banach space theory and set-theoretic topology we will adopt as in \cite{Biorthog}. 

\section{Complementations in $C(K,X)$}
\label{Sec-C(K,X)} 

Our first main result is the following Cembranos-Freniche type theorem obtained by joining ideas of \cite[Corollary 4.6]{GaHa} and \cite[Proposition 5.5]{AlspachGalego}. 

\begin{theorem}\label{mainTheorem1} Let $K$ be a Hausdorff compactum and $X$ be a Banach space. Assume, for some infinite cardinal $\kappa$, a collection $(K_\alpha)_{\alpha<\kappa}$ of pairwise disjoint non-empty clopen subsets of $K$, and a family of projections $(P_\alpha)_{\alpha<\kappa}$ defined on $X$, $P_\alpha[X]=X_\alpha$, such that the formula $x \mapsto (P_\alpha(x))_{\alpha<\kappa}$ defines a bounded linear operator  from $X$ to 
$\left(\bigoplus_{\alpha<\kappa}X_{\alpha}\right)_{c_0}$. Then $C(K,X)$ has a complemented subspace isomorphic to $\left(\bigoplus_{\alpha<\kappa}C(K_{\alpha},X_{\alpha})\right)_{c_0}$.
\end{theorem}
\begin{proof} Let $Y=\left(\bigoplus_{\alpha<\kappa}C(K_{\alpha},X_\alpha)\right)_{c_0}$. Let $S:C(K,X)\to Y$ be the map defined by the formula
\[S(f)=(P_\alpha \circ f\restriction_{K_{\alpha}})_{\alpha<\kappa}.\]

From the assumptions on $(P_\alpha)_{\alpha<\kappa}$ and the Banach–Steinhaus theorem, it is readily seen that $S$ is a well defined linear operator with $\|S\|=\sup_{\alpha<\kappa}\|P_\alpha\|<\infty$. 

On the other hand, for each $g=(g_\alpha)_{\alpha<\kappa}\in Y$ let $\varphi_g:K\to X$ be defined by 
\begin{align*}
\varphi_g(t)= \left\{
\begin{array}{ll}
0 &\text{ if }t\in K\setminus \cup_{\alpha<\kappa}K_{\alpha}.\\
g_{\alpha}(t) &\text{ if }t\in K_{\alpha}.
\end{array} \right.
\end{align*}

We claim that each $\varphi_{g}$ is continuous in $K$. Indeed, let $t\in K$ be arbitrary and let $(t_j)_{j\in J}$ be any net in $K$ converging to $t$. Since each $K_{\alpha}$ is a clopen set, from the continuity of each $g_\alpha$ it is evident that 
$\varphi_{g}(t_j)\to \varphi_g(t)$ if $t\in \cup_{\alpha<\kappa}K_{\alpha}$. Suppose that $t\in K\setminus \cup_{\alpha<\kappa}K_{\alpha}$ and let $\epsilon>0$ be arbitrary. Recalling that $(g_{\alpha})_{\alpha<\kappa}\in Y$, the set $F=\{\alpha<\kappa:\|g_{\alpha}\|\geq \epsilon\}$ is finite. There is $j_0\in J$ such that $t_j\in K\setminus \bigcup_{\alpha\in F}K_{\alpha}$ whenever $j\geq j_0$. Then, $\|\varphi_g(t_j)\|<\epsilon$ whenever 
$j\geq j_0$ and we deduce that $\varphi_{g}(t_j)\to \varphi_{\alpha}(t)=0$, establishing our claim. We define an operator $T:Y\to C(K,X)$ by $T(g)= \overline{\varphi_g}$. It is easily seen that $T$ is a well defined linear operator. Moreover, $\|T\|=1$.

It is now immediate from the construction of $S$ and $T$ that 
\[S\circ T(g)=(P_\alpha\circ T(g)\restriction_{K_{\alpha}})=(g_\alpha)_\alpha=g,\]
for each $g=(g_\alpha)_\alpha\in Y$. From Proposition \ref{fact1} we deduce that $T\circ S$ is a projection of $C(K,X)$ onto an isomorphic copy of $Y$.
\end{proof}

Given a cardinal $\kappa$, see \cite[Definition 4.3]{GaHa}, we recall the concept of $\kappa$-Josefson-Nissenzweig property (in short, $JN_\kappa$ property). A Banach 
space $X$ is said to have the $JN_\kappa$ property whenever admits a $JN_\kappa$ family, that is, a family of functionals $(\varphi_{\alpha})_{\alpha<\kappa}$ in $S_{X^*}$ such that $(\varphi_\alpha(x))_{\alpha<\kappa}\in c_0(\kappa)$ for every $x\in X$. The name of this property comes from the celebrated theorem of Josefson-Nissenzweig, see \cite[Chapter XII]{Diestel}, that in our terminology states that every infinite dimensional Banach has the $JN_\omega$ property. 

\begin{remark}In contrast with Josefson-Nissenzweig theorem, Galego and Hagler in \cite{GaHa} observed that, by a result of Todor\v cevi\'c \cite[Corollary 6]{Todor} obtained under the assumption of an extra set-theoretic assumption, it is relatively consistent with ZFC that every Banach space $X$ with density character $\omega_1$ has the $JN_{\omega_1}$ property. 
In the opposite direction, under the assumption of another extra set-theoretic principle, it is possible to construct Banach spaces of the form $C(K)$ with density character $\omega_1$
that have not the $JN_{\omega_1}$ property, see \cite{CK} and \cite[Theorem 4.2 and Remark 4.3]{candido2}.
\end{remark}

\begin{proposition}\label{auxprojection}Let $X$ be a Banach space and $\kappa$ be an infinite cardinal. If $X$ has the $JN_\kappa$ property, then there is a family of projections 
$(P_\alpha)_{\alpha<\kappa}$ defined in $X$, such that $P_{\alpha}[X]=X_\alpha\cong \R$ and the formula $x \mapsto (P_\alpha(x))_{\alpha<\kappa}$ defines a bounded linear operator  from $X$ to $\left(\bigoplus_{\alpha<\kappa}X_{\alpha}\right)_{c_0}\cong c_0(\kappa)$.
\end{proposition}
\begin{proof}
Let $(\varphi_{\alpha})_{\alpha<\kappa}$ be a $JN_\kappa$ family for $X$. For each $\alpha<\kappa$ let $x_\alpha\in X$ be such that $\|x_\alpha\|<2$ and $\varphi_{\alpha}(x_{\alpha})=1$ and define $P_\alpha:X\to X$ by $P_\alpha(x)=\varphi_{\alpha}(x)\cdot x_{\alpha}$. It is clear that 
$(P_{\alpha})_{\alpha<\kappa}$ is a family of projections of $X$ onto $P_{\alpha}[X]=\mathrm{span}\{x_\alpha\}\cong \R$ and, by the properties of $(\varphi_{\alpha})_{\alpha<\kappa}$, $X\ni x\mapsto (P_{\alpha}(x))_{\alpha<\kappa}\in\left(\bigoplus_{\alpha<\kappa}X_\alpha\right)_{c_0}$ is a well defined bounded linear operator.
\end{proof}

From Theorem \ref{mainTheorem1} and Proposition \ref{auxprojection} we obtain:

\begin{corollary}\label{VersionGaHa}Let $K$ be Hausdorff compactum and let $(K_{\alpha})_{\alpha<\kappa}$ be a collection of pairwise disjoint non-empty clopen compact subsets of $K$. If $X$ is a Banach space having the $JN_\kappa$ property, then $C(K,X)$ has a complemented subspace isomorphic to $c_0(C(K_\alpha))$.
\end{corollary}

Recalling that a $C(K)$ space has a subspace isomorphic to $c_0(\kappa)$ if and only if $K$ has a collection of cardinality $\kappa$, consisting of pairwise disjoint non-empty open sets, Corollary \ref{VersionGaHa} provides an extension of \cite[Corollary 4.6]{GaHa} for zero-dimensional compacta.

\begin{corollary} 
$C(\beta \N^*\times \beta \N)\cong C(\beta\N^*,\ell_{\infty})\cong  C(\beta\N,\ell_{\infty}/c_0)$ has a complemented subspace isomorphic to $c_0(\mathfrak{c},\ell_{\infty}/c_0)$.
\end{corollary}
\begin{proof}
Let $\{A_\alpha:\alpha<\mathfrak{c}\}$ be an almost disjoint family of subsets of $\N$. That is, for each $\alpha, \beta<\mathfrak{c}$, $|A_{\alpha}|=|A_{\beta}|=\omega$ and $|A_{\alpha}\cap A_{\beta}|<\omega$. The collection $\{V_\alpha:\alpha<\mathfrak{c}\}$, where $V_\alpha=\overline{A_\alpha}^{\beta\N}\setminus \N$ for each $\alpha$, constitutes a collection of pairwise disjoint clopen sets of $\beta\N^*$, each of them homeomorphic to $\beta \N^*$. In addition, from \cite[Theorem 4.22]{Biorthog}, $\ell_2(\mathfrak{c})$ is a quotient of $C(\beta \varGamma)$, consequently, $\ell_{\infty}\cong C(\beta\N)$ has $JN_{\mathfrak{c}}$ property. From Corollary \ref{VersionGaHa} we obtain $c_0(\mathfrak{c},\ell_{\infty}/c_0)\cong \left(\bigoplus_{\alpha<\mathfrak{c}}C(V_\alpha)\right)_{c_0}\complemented C(\beta\N^*,\ell_{\infty}).$
\end{proof}

Another application of the previous results leads to the following interesting corollary:

\begin{corollary}\label{betaKcomplemented}
For an infinite family $(K_\alpha)_{\alpha<\kappa}$ of locally compact Hausdorff spaces, let $K=\amalg_{\alpha<\kappa}K_\alpha$ be its disjoint union endowed with topology $\{ U\subset K: U\cap K_\alpha\text{ is open in }K_\alpha\text{ for all }\alpha<\kappa\}$ and let $X$ be a Banach space with the $JN_\kappa$ property. 
Then $C(\beta K,X)$ has a complemented subspace isomorphic to $\left(\bigoplus_{\alpha<\kappa} C(\beta K_\alpha)\right)_{c_0}$. 
\end{corollary}
\begin{proof}
It is clear that $K$ is a locally compact Hausdorff space whence we may consider its Stone-\v Cech compactification $\beta K$. Since each $K_\alpha$ is a clopen subset of $K$ we have $\overline{K_{\alpha}}^{\beta\varGamma}\approx \beta K_\alpha$ for each $\alpha<\kappa$, moreover,  $\{\overline{K_{\alpha}}^{\beta K}:\alpha<\kappa\}$ constitutes a collection of pairwise disjoint clopen sets of $\beta K$, see \cite[Theorem 6.5]{GillJer}. Then, since $X$ has the $JN_\kappa$ property, Corollary \ref{VersionGaHa} gives that $\left(\bigoplus_{\alpha<\kappa}C(\beta K_{\alpha})\right)_{c_0}\complemented C(\beta K, X)$.
\end{proof}

\begin{remark} Let $\varGamma$ be an infinite set with $|\varGamma|=\kappa$. Because $\ell_2(2^\kappa)$ is a quotient of $C(\beta \varGamma)$, see  \cite[Theorem 4.22]{Biorthog}, $C(\beta\varGamma)$ admits a $JN_\kappa$ family. Applying Corollary \ref{betaKcomplemented} with $K=\varGamma$, $(K_\alpha)_{\alpha<\kappa}$ a partition of $\varGamma$ into pairwise disjoint subsets of cardinality $\kappa$ and $X=C(\beta\varGamma)$ we deduce that $C(\beta \varGamma,C(\beta\varGamma))\cong C(\beta \varGamma\times \beta\varGamma)$ has a complemented subspace isomorphic to $\left(\bigoplus_{\alpha<\kappa}C(\beta \varGamma_\alpha)\right)_{c_0}\cong c_0(\varGamma,C(\beta\varGamma))$.
\end{remark}

\begin{remark}
It is important to observe that the result of the previous remark, for the case $\varGamma=\N$, even though were not explicitly stated in \cite{AlspachGalego}, can be derived from \cite[Proposition 5.5]{AlspachGalego}.
\end{remark}

\section{Complementations in $\ell_{\infty}(C(K))$}
\label{Sec-ellinfty(X)}

For a topological space $K$ and a Banach space $X$, $C_b(K,X)$ denotes the Banach space of all bounded continuous functions $f:K\to X$, equipped with the norm 
$\|f\|=\sup_{t\in K}\|f (t)\|$. For simplicity, $C_b(K,\R)$ will be denoted by $C_b(K)$. It is evident that $C_b(\N,X)=\ell_{\infty}(X)$ and $C_b(\N,\R)=\ell_{\infty}$.

The next result plays a central role in our investigation. The idea comes from a theorem of Khmyleva, see \cite[Theorem 1]{Khmyleva}.

\begin{theorem}\label{mainTheorem2}
Let $K$ be a topological space and $Y$ be a Banach space. Assume that there is a sequence $(K_n)_n$ of subsets of $K$ and a sequence of Banach spaces $(X_n)_n$ satisfying the following conditions:
\begin{itemize}
    \item[(1)]$\bigcup_{n\in \N} K_n=K$.
    \item[(2)]$K_n\subset \mathrm{int}(K_{n+1})$ for each $n \in \N$.
    \item[(3)]$X_n\not\hookrightarrow C_b(K_n,Y)$ for each $n \in \N$.
\end{itemize}
Then, $C_b(K,Y)$ has no complemented subspace isomorphic to $\left(\bigoplus_{n\in \N}X_n\right)_{c_0}$.
\end{theorem}
\begin{proof}
Let $E=\left(\bigoplus_{n\in \N}X_n\right)_{c_0}$. Towards a contradiction, assume that there are linear operators $T:E\to C_b(K,X)$ and 
$S:C_b(K,X)\to E$ such that $S\circ T=\Id_E$ where $\Id_E$ is the identity map on $E$. From Proposition \ref{fact1} 
we know that $T$ is a linear embedding onto a complemented subspace of $C_b(K,Y)$, isomorphic to $E$. Let 
$C>0$ such that $\|T(u)\|\geq C\|u\|$ for each $u\in E$.

For each $n \in \N$, let $I_n:X_n\to E$ be the canonical embedding of $X_n$ into $E$ and let $R_n:C_b(K,Y)\to C_b(K_n,Y)$ the restriction map $f\mapsto f\restriction_{K_n}$.

From condition (3), the map $R_n\circ T \circ I_n:X_n\to C_b(K_n,Y)$ is not a linear embedding. Then, there exists 
$x_n \in X_n$ such that $\|x_n\|=1$ and 
\[\sup_{t\in K_n}\|T(I_n(x_n))(t)\|<2^{-n}.\]

For each $a=(a_n)_n \in \ell_\infty$, let $f_a:K\to Y$ be defined by
\[f_a(t)=\sum_{n=1}^{\infty} a_n T(I_n(x_n))(t).\]

Given $t\in K$, from condition (1) we have $A_t=\{n\in \N: t\in K_n\}\neq \emptyset$, and we may fix $m_t=\min (A_t)$.
Then, for every $m\geq m_t$, we have the following relation
\begin{align*}
\label{Mtest}\|f_a(t)\|&\leq \|\sum_{n< m}a_nT(I_n(x_n))(t)\|+\sum_{n\geq m}|a_n|2^{-n}\\
                     &\leq \|T(\sum_{n< m}a_nI_n(x_n))\|+\frac{\|a\|_\infty}{2^{m-1}}\leq (\|T\|+\frac{1}{2^{m-1}})\|a\|_{\infty}.
\end{align*}

And we deduce that $f_a$ is a well defined bounded map satisfying \[\sup_{t\in K}\|f_a(t)\|\leq \|T\|\|a\|_{\infty}.\] 
Furthermore, from 
Weierstrass's M-test it is readily seen that $f_a$ is continuous in each $K_n$. Then, by condition (2), $f_a$ is continuous in all $K$. 

It follows that the linear map $Q:\ell_\infty\to C_b(K,Y)$, $Q(a)=f_a$, is bounded $\|Q\|\leq \|T\|$. Moreover, if $(e_n)_n$ denotes the canonical unit sequence in $\ell_\infty$, then for each for each $n\in\N$ we have $\|Q(e_n)\|=\|T(I_n(x_n))\|\geq C>0$. Then, according to 
\cite[Theorem 7.10]{Biorthog}, there is a subsequence $(x_{n_r})_r$ such that $\ell_{\infty}\ni (a_r)_r\mapsto \sum_{r=1}^\infty a_rT(x_{n_r})\in C_b(K,X)$ is a linear embedding, an isomorphism onto its image $Z=\{\sum_{r=1}^\infty a_rT(x_{n_r}):(a_r)_r\in \ell_\infty\}$.

Let $Z_0=\{\sum_{r=1}^\infty b_rI_{n_r}(x_{n_r}): b=(b_r)_r\in c_0\}$. It is evident that $Z_0$ is a subspace of $E$, isomorphic to $c_0$.
For each $r\in \N$, let $\varphi_{n_r}\in 2B_{X_{n_r}^*}$, such that $\varphi_{n_r}(x_{n_r})=1$.

Let $P:E\to E$ be the operator defined by 
\[P((z_n)_n)=\sum_{r\in \N}\varphi_{n_r}(z_{n_r})I_{n_r}(x_{n_r}).\]  

It is evident that $P$ is a projection of $E$ onto $Z_0$. Recalling the map $S:C_b(K,Y)\to E$ from the beginning, we have
$(P \circ (S\restriction_{Z}))\circ T\restriction_{Z_0}=\Id_{Z_0}$. From Proposition \ref{fact1} we deduce that $Z\sim\ell_\infty$ has a complemented subspace isomorphic to $Z_0\sim c_0$, a  contradiction.

\end{proof}

\begin{corollary}\label{cor1}
Let $X$ and $Y$ be Banach spaces such that $c_0(X)\not\hookrightarrow Y^n$ for every $n \in \N$. Then $\ell_{\infty}(Y)$ has no complemented subspace isomorphic to 
$c_0(X)$. 
\end{corollary}
\begin{proof}
By applying Theorem \ref{mainTheorem1} with $K=\N$ and, for each $n\in \N$, $K_n=\{1,2,3,\ldots,n\}$ and $X_n=c_0(X)$, we deduce that
$\ell_{\infty}(Y)$ has no complemented subspace isomorphic to $c_0(c_0(X))\sim c_0(X)$. 
\end{proof}

\begin{remark}
The previous corollary should be compared with the well known fact that, for every Banach space 
$X$, $(\ell_{\infty}(X))^*$ has a complemented subspace isomorphic to $\ell_1(X^*)\cong (c_0(X))^*$. 
Indeed, let $T:\ell_1(X^*)\to (\ell_{\infty}(X))^*$ be the linear operator given by the formula 
$T((\varphi_n)_n)((x_n)_n)=\sum_{n=1}^{\infty} \varphi_n(x_n)$ and
$S:(\ell_{\infty}(X))^*\to \ell_1(X^*)$, $S(\Phi)=(\Phi\circ I_n)_n$, where $I_n:X\to \ell_{\infty}(X)$
denotes the canonical embedding of $X$ into the $n$-coordinate of $c_0(X)$.

It is evident that $S\circ T=\Id_{\ell_1(X^*)}$ and consequently, by Proposition \ref{fact1}, 
$(\ell_{\infty}(X))^*$ has a complemented copy of $\ell_1(X^*)$.

On the other hand, it is not true in general that $\ell_{\infty}(X)$ has a complemented subspace isomorphic to $c_0(X)$. If $X$ is finite dimensional, then  $c_0\sim c_0(X)\not\complemented \ell_{\infty}(X)\sim \ell_\infty$. With the help of Corollary \ref{cor1} we may find many such examples of infinite dimension. For example,
if $1\leq p<\infty$ then, for each $n \in \N$, $c_0(\ell_p)\not\hookrightarrow \ell_p^n\sim \ell_p$. Therefore, $c_0(\ell_p)\not\complemented \ell_{\infty}(\ell_p)$.
For an interesting non-separable example, let $X=C([0,\omega_1])$. According to \cite[Corollary 1.8]{KKL}, $c_0(C([0,\omega_1]))\not\hookrightarrow C([0,\omega_1])^n$ for each $n \in \N$. From Corollary \ref{cor1} we may deduce that $c_0(C([0,\omega_1]))\not\complemented \ell_{\infty}(C([0,\omega_1]))$.
\end{remark}


Concerning the Problem \ref{mainproblem} we have the following result. It is related to the main results of \cite{Candido1} and \cite{candido3}.

\begin{theorem}\label{Scattered}
Let $X$ and $Y$ be Banach spaces having no subspace isomorphic to $c_0$ and let $K$ and $L$ be compact Hausdorff spaces such that 
$C(K)\sim c_0(C(K))$ and $C(L)\sim c_0(C(L))$. Then 
\[\ell_{\infty}(C(K,X))\sim \ell_{\infty}(C(L,Y))\Rightarrow |K|=|L|.\] 
Moreover, $L$ is scattered if and only if $K$ is scattered.
\end{theorem}
\begin{proof}
From the hypotheses we have 
\[c_0(C(K))\sim C(K)\complemented C(K,X)\complemented \ell_{\infty}(C(K,X))\sim \ell_{\infty}(C(L,Y)).\]
From Corollary \ref{cor1}, there is $n\in \N$ such that $C(K)\hookrightarrow C(L,Y)^n$. Recalling that  $C(L)\sim c_0(C(L))$, we have
\[C(K)\hookrightarrow C(L,Y).\]
From \cite[Theorem 1.3]{Candido1} we deduce that $|K|\leq |L|$. Furthermore, from \cite[Theorem 1.3]{Candido1}, if $L$ is scattered, then $K$ is scattered.
We are done by changing the roles of $K$ and $L$ in the foregoing argument.

\end{proof}

\begin{remark}We observe that Theorem \ref{Scattered} is false without the hypotheses $C(K)\sim c_0(C(K))$ and $C(L)\sim c_0(C(L))$.
Recalling that $\ell_{\infty}(C(\beta\N))\cong \ell_{\infty}(\ell_{\infty})\cong \ell_{\infty}$ and $c_0\sim C([0,\omega])$, and denoting by $\beta\N \dotcup [0,\omega]$ the disjoint 
union of $\beta\N$ and $[0,\omega]$, we have
\[\ell_{\infty}(C(\beta\N \dotcup [0,\omega]))\sim\ell_\infty(C(\beta\N))\oplus C([0,\omega]))\sim \ell_{\infty}(\ell_{\infty})\oplus \ell_{\infty}(C([0,\omega]))\sim \ell_{\infty}(C([0,\omega])).\]

However 
\[|[0,\omega]|=\omega<2^{2^\omega}=|\beta\N \dotcup [0,\omega]|.\] 
Furthermore, $[0,\omega]$ is scattered while $\beta\N \dotcup [0,\omega]$ is not.

\end{remark}

The following lemma relates to \cite[Theorem 3.5]{AlspachGalego}. We use a similar idea as in \cite[Lemma 2.3]{Candido1}.

\begin{lemma}
Let $K$ and $L$ be infinite metric compacta and let $X$ be a Banach space containing no subspace isomorphic to $c_0$. Then
\[C(K)\hookrightarrow C(L,X)\Rightarrow C(K)\complemented C(L)\]
\end{lemma}
\begin{proof}
If $L$ is uncountable, then by Miljutin theorem \cite{Mil}  \cite[Theorem 21.5.10]{Se}, $C(L)\sim C([0,1])$. From Banach-Mazur theorem \cite[Theorem 8.7.2]{Se}  we have $C(K)\hookrightarrow C(L)$, then, by \cite[Theorem 1]{SuperPel}, $C(K)\complemented C(L)$. 

On the other hand, suppose that $L$ is countable and let $Y$ be a subspace of $C(L, X)$ that is isomorphic to $C(K)$. Since $Y$ is separable and $\{f \cdot u : (f,u) \in B_{C(L)}\times B_X\}$ spans
a dense subset in $C(L, X)$, there is $\{f_n : n \in \N\} \subset B_{C(L)}$, and $\{u_n : n \in \N\}\subset B_X$ such that 
$Y\subset \overline{\mathrm{span}}(\{f_n\cdot u_m : n, m \in \N\})$. Let $L_0 = L/\sim$ be the quotient space, where $x\sim y$ if and only if $f_n(x)=f_n(y)$ 
for every $n$, and let $q : L \to L_0$ be the quotient map. It is readily seen that $L_0$ is a countable metrizable compactum, and the formula $g\mapsto g\circ q$ 
defines an isometric linear embedding of $C(L_0)$ into $C(L)$. From \cite[Theorem 1]{SuperPel} we have $C(L_0)\complemented C(L)$.

Let $X_0 = \overline{\mathrm{span}}(\{u_n : n\in \N\})$ and for each $n \in \N$, let $g_n :L_0\to \R$ be given by $g_n(q(x)) = f_n(x), x\in L$. Clearly, 
$g\in C(L_0)$ and $\mathrm{span}(\{f_n\cdot u_m :n, m \in \N\}) \sim \mathrm{span}(\{g_n\cdot u_m : n, m\in \N\}) \subset C(L_0,X_0)$. We deduce
\[C(K)\sim Y\hookrightarrow C(L_0,X_0).\]

Since $C(L_0,X_0)$ is separable, by \cite[Theorem 1]{SuperPel}, we have
\[C(K)\complemented C(L_0,X_0)\]

Because $L_0$ is a countable metric compactum, it is a compact scattered first-countable space. By Mazurkiewicz-Siepi\'nski \cite[Theorem 8.6.10]{Se} and Bessaga-Pelczy\'nski \cite[Theorem 1]{BP}, there is $0 \leq \alpha < \omega_1$ such that $C(L_0)\sim C([0,\omega^{\omega^\alpha}])$. Hence
\[C(K)\complemented C(L_0,X_0)\sim C([0,\omega^{\omega^\alpha}],X_0)\complemented C(\beta\N,C([0,\omega^{\omega^\alpha}],X_0))\cong C(\beta\N\times [0,\omega^{\omega^\alpha}],X_0)\]

Since $X_0$ has no subspace isomorphic to $c_0$, from \cite[Theorem 3.5]{AlspachGalego} we deduce
\[C(K)\complemented C([0,\omega^{\omega^\alpha}])\sim C(L_0).\]

Recalling that $C(L_0)\complemented C(L)$ we conclude $C(K)\complemented C(L)$.

\end{proof}

From the previous lemma and Pe{\l}czy\'nski decomposition method we obtain the following extension of \cite[Theorem 3.8]{AlspachGalego}.

\begin{proposition}\label{extensionfalspachgalego}Let $X$ and $Y$ be a Banach spaces containg no subspace isomorphic to  $c_0$.
Then for any infinite metric compacta $K_1$ and $K_2$,
\[C(K_1)\hookrightarrow C(K_2,Y)\text{ and } C(K_2)\hookrightarrow C(K_1,X) \Rightarrow  C(K_1)\sim C(K_2)\].
\end{proposition}

From \cite[Corollary 5.8]{galego1} (also \cite[Corollary 3.8]{AlspachGalego}) we know that for infinite metric compacta $K_1$ and $K_2$, 
$C(K_1,\ell_{\infty})\sim C(K_2,\ell_{\infty})$ if and only if $C(K_1)\sim C(K_2)$. We will see in the next section that whenever $K$ is a compact metric space,
$C(K,\ell_{\infty})\not\sim \ell_{\infty}(C(K))$. In addition, we have the following complementary result.

\begin{theorem}Let $X$ and $Y$ be a Banach spaces containg no subspace isomorphic to  $c_0$.
Then for any infinite metric compacta $K_1$ and $K_2$,
\[\ell_{\infty}(C(K_1,X))\sim \ell_{\infty}(C(K_2,Y))\Rightarrow C(K_1)\sim C(K_2).\]
\end{theorem}
\begin{proof}
If $\ell_{\infty}(C(K_1,X))\sim \ell_{\infty}(C(K_2,Y))$ then 
\[C(K_1)\complemented \ell_{\infty}(C(K_2,Y))\text{ and }C(K_2)\complemented \ell_{\infty}(C(K_1,X)).\]

Since $c_0(C(K_i))\sim C(K_i)^n\sim C(K_i)$ for each $i=1,2$ and for every $n \in \N$, we deduce from Corollary \ref{cor1} that
$C(K_1)\hookrightarrow C(K_2,Y)$ and $C(K_2)\hookrightarrow C(K_1,X)$. Then, by Proposition \ref{extensionfalspachgalego}, we have
$C(K_1)\sim C(K_2)$.

\end{proof}

\section{Further Applications}
\label{Sec-FurtherApp}

A natural question related to this this investigation is whether exists a Hausdorf compactum $K$ such that 
\begin{equation}\label{rel1}
C(K,\ell_{\infty})\sim \ell_{\infty}(C(K)).
\end{equation}

Since $\ell_{\infty}\cong C(\beta \N)$, we have $C(K,\ell_{\infty})\cong C(K\times \beta \N)\cong C(\beta\N,C(K))$, and the previous question could be put into words in the following way: 

\begin{problem}Is there a Hausdorff compactum $K$ such that the space of bounded sequences in $C(K)$ is isomorphic to the space of precompact sequences in $C(K)$?
\end{problem}

As observed by D. Alspach and E. M. Galego in \cite[p.157]{AlspachGalego}, if such infinite $K$ exists, then $C(K)$ has a complemented subspace isomorphic to $c_0$. Indeed, by \cite{Cem}, $c_0\complemented C(\beta\N,C(K))\sim \ell_{\infty}(C(K))$ and, by \cite{LeungRabi}, $c_0\complemented \ell_\infty(C(K))$ if and only if $c_0\complemented C(K)$. We deduce, by \cite[Corollary 2]{Cem}, that if $C(K)$ is a Grothendieck space, 
then $C(K,\ell_{\infty})\not\sim \ell_{\infty}(C(K))$.

By using some of the results obtained in our investigation we can add more information.

\begin{theorem}\label{isoiso}
Let $K$ be an infinite Hausdorff compactum. Then 
\[C(K,\ell_{\infty})\sim \ell_{\infty}(C(K))\Rightarrow \ell_{\infty}\hookrightarrow C(K).\]
\end{theorem}
\begin{proof}
Suppose that $C(K,\ell_{\infty})\sim \ell_{\infty}(C(K))$. Since $C(K,\ell_{\infty})\sim C(\beta\N, C(K))$ and $C(K)$ is infinite dimensional, by Corollary 3.6,  $C(\beta\N,C(K))$ has a complemented subspace isomorphic to $c_0(C(\beta\N))$. Hence 
$c_0(C(\beta\N))\complemented  \ell_{\infty}(C(K))$. According to Corollary \ref{cor1}, there is $n \in \N$ such that $\ell_{\infty}\hookrightarrow C(K)^n$.
Then, by \cite{Drew}, $\ell_{\infty}\hookrightarrow C(K)$.
\end{proof}

\begin{remark}
From Theorem \ref{isoiso} we deduce that $\ell_{\infty}(C(K))\not\sim C(K,\ell_{\infty})$ wherever $K$ is a metric compactum. 
This extends the main result of \cite{CemMend} and \cite[Corollary 3]{Khmyleva}. It would be interesting to know whether there exists an infinite 
Hausdorff compactum $K$ such that (\ref{rel1}) holds.
\end{remark}

\section{Acknowledgements}

The author was supported by Funda\c c\~ao de Amparo \`a Pesquisa do Estado de S\~ao Paulo - FAPESP No. 2016/25574-8

\bibliographystyle{amsalpha}

\end{document}